\documentclass[10pt]{amsart}
\usepackage[utf8]{inputenc}
\usepackage{amsmath, amssymb, amsthm, amsfonts, bm}
 \newtheorem{theorem}{Theorem}
 
 \newtheorem{proposition}[theorem]{Proposition}
 \newtheorem{corollary}[theorem]{Corollary}

 \theoremstyle{definition}
 \theoremstyle{remark}
\newtheorem{remark}[theorem]{Remark}

\theoremstyle{definition}
\newtheorem{definition}[theorem]{Definition}
\usepackage{amssymb}
\usepackage{tikz,lipsum,lmodern}
\usepackage[most]{tcolorbox}
\usepackage{xcolor}
\usepackage{color}
\usepackage{amsmath}
\usepackage{amsthm}
\usepackage{comment}
\newcommand{\GL}{\textnormal{GL}}

\newcommand{\mbb}{\mathbb}
\newcommand{\N}{\mathbb N}
\newcommand{\Z}{\mathbb Z}
\newcommand{\R}{\mathbb R}

\let\ol=\overline

\newcommand{\T}{\overline{T}}
\newcommand{\uu}{\overline{u}}
\newcommand{\vv}{\overline{v}}

\usepackage{color}
\usepackage{hyperref}
\hypersetup{
    colorlinks=true, 
    linktoc=all,     
    linkcolor=blue,  
    citecolor=blue,
    filecolor=blue,
    urlcolor=blue
}

\title[On the distality and expansivity of certain maps on spheres]{On the distality and expansivity of certain maps on spheres}

\author{Manoj Choudhuri}
\address{Institute of Infrastructure, Technology, Research and Management, Near Khokhara Circle, maninagar (East), Ahmedabad 380026, Gujarat, India.}
\email{manojchoudhuri@iitram.ac.in}

\author{Gianluca Faraco}
\address{\textsuperscript{b}Dipartimento di Matematica e applicazioni, Università degli studi di Milano-Bicocca, via Cozzi 55, Milano 20125, Italy}
\email{gfaraco@uni-bonn.de}

\author{Alok Kumar Yadav}
\address{Department of Mathematics, Sikkim University, Sikkim 737102, India}
\email{akyadav@cus.ac.in}
\date{}

\begin{document}


\begin{abstract}
Any affine map on the (n+1)-dimensional Euclidean space gives rise to a natural map on the n-dimensional sphere whose dynamical aspects are not so well-studied in the literature. We explore the dynamical aspects of these maps  by investigating about their distality and expansivity.
\end{abstract}

\subjclass[2020] {Primary:37B05; 37C25; Secondary:37E10}
\keywords{Topological dynamics; homeomorphisms; spheres; distality; expansivity}

\maketitle
\tableofcontents
\section{Introduction}

The general theory of $1$-dimensional dynamical system is well developed through time. In particular, the dynamics of homeomorphisms of circle drew attention of many mathematicians over the years since its introduction by Henri Poincar\'e in 1882 \cite{HP}, see also \cite{AK, HMR, Ma, Wa, Zd} and references cited therein for various results related to circle homeomorphisms. In this article, we are mainly interested in the dynamics of a particular type of homeomorphisms on the circle and higher dimensional spheres which is not so well-studied in the literature. We are mainly going to investigate about the distality and expansivity of these maps.
Given a metric space $(X,d)$, a transformation $T\in\textnormal{Homeo}(X)$ is said to be {\it distal} if for any pair of distinct elements $x, y\in X$, the closure of the double orbit $\{(T^n(x),T^n(y))\mid n\in\Z\}$ does not intersect the diagonal $\{(c,c)\mid c\in X \}$. On the other hand, $T$ is said to be expansive with expansive constant $\delta>0$ if for each pair $(x,y)$ of distinct points of $X$, there is an integer $n$ such that $d(T^n(x),\ T^n(y))>\delta$.
\smallskip

The notion of distality was first introduced by Hilbert, \textit{cf.}\ Ellis \cite{E4}, Moore \cite{M9}, and studied in different contexts by many people, see for instance Abels \cite{A1,A2}, Furstenberg \cite{F6}, Raja-Shah  \cite{RaSh10,RaSh11} and Shah \cite{Sh12}, and references cited therein. 
In this article, we are interested in a family of maps from $\mathbb{S}^n$, the $n$-dimensional sphere, to itself which has been recently introduced by Shah-Yadav in \cite{SY1}. Let $T\in \GL(n+1,\R)$ and $a\in\R^{n+1}\setminus\{0\}$, the map $\T_a\colon \mbb{S}^n\to\mbb{S}^n$ is defined as 
\begin{equation}\label{tabar}
    \T_a(x)=\frac{a+T(x)}{\|a+T(x)\|},
\end{equation}
$\|\,\cdot\, \|$ being the Euclidean norm on $\R^{n+1}$.
  Such a map is clearly continuous and turns out to be a homeomorphism whenever $\|T^{-1}(a)\|<1$, see \cite{SY1} for details.
   For certain classes of $\T_a$, it was shown in \cite{SY1} that these maps are not distal.
  In this article, we obtain a more general result. In fact, we show that given any $T\in \GL(n+1,\mathbb{R})$ there exists a conjugate of $T$, say $S$, and $a\in\R^{n+1}$, such that $\ol{S}_a$ is not distal on $\mbb{S}^n$. Of course, in many cases $\T_a$ itself is not distal. While discussing the dynamics of $\T_a$ on $\mbb{S}^1$, we study the existence of fixed and periodic points of $\T_a$ on $\mathbb{S}^1$ and the behaviour of the orbits of points in a deleted neighbourhood of those fixed and periodic points. As a consequence, we are able to say that the map $\T_a$ is not distal in many cases. The reader is referred to \cite{SY2} for the study of distality of a similar type of maps on $p$-adic spheres.  
\smallskip

It is worth mentioning that 
given a linear map $T\in \GL(n+1,\R)$, the dynamics of an affine transformation $T_a=T(x)+a$ on $\R^{n+1}$ (seen as an affine space), with $a\in \R^{n+1}$ (\textit{cf.} \cite{M9}), and the dynamics of $\T_a$ on $\mbb{S}^n$ can be quite different in nature. For example, if $T\in \GL(n+1,\R)$ has all the eigenvalues of absolute value $1$, then it follows from \cite[Theorem 1]{M9} that $T_a$ is distal for any $a\in\R^{n+1}$. On the other hand, if 
\begin{equation*} T=
   \begin{pmatrix}
   1 & 1\\
   0 & 1
   \end{pmatrix}, 
\end{equation*} and $a=0$, then the action of $\T_a$ on $\mbb{S}^1$ is not distal, see \cite[Section 2]{SY1} for details.

\smallskip

Like distality, expansivity has also been studied in various contexts over the years. In the case of connected locally compact groups, the existence of expansive automorphisms has certain implications on the structure of the group, see \cite{CR1}, \cite{Sh13} and the references cited therein for more details and related results in this context. In this article, we are concerned with expansive homeomorphisms of spheres. It is shown in \cite{JU1} and \cite{R1} that there can not be an expansive homeomorphism on $\mbb{S}^1$. It is also known that there is no expansive homeomorphism on $\mbb{S}^2$ as well, see \cite{Hi}, \cite{L2} for details, though expansive homeomorphisms exist for surfaces of genus greater or equal to $1$. One may look at \cite{Hi}, \cite{L2} and the references therein for more details about expansive homeomorphisms on surfaces and its relation to Anosov diffeomorphisms and pseudo-Anosov maps. In this article, we show that for each $T\in \GL(n+1,\R)$ with $n>2$, there exist uncountably many non-zero $a\in\R^{n+1}$ such that $\T_a$ is not expansive on $\mbb{S}^n$.

\section{Dynamics on $\mbb{S}^1$}\label{cc}
\begin{definition}
Let $X$ be a locally compact metric space and $f\colon X\to X$ a continuous map. A fixed point $p$ of $f$ is {\it attracting} if it has a neighbourhood $U$ such that $\overline{U}$ is compact, $f(\overline{U})\subset U$, and $\bigcap\limits_{n\geq 0} f^n(U)=\{p\}$. A fixed point $p$ is {\it repelling} if it has a neighbourhood $U$ such that $\overline{U}\subset f(U)$, and $\bigcap\limits_{n\geq 0} f^{-n}(U)=\{p\}$. Note that if $f$ is invertible, then $p$ is an attracting fixed point for $f$ if and only if it is a repelling fixed point for $f^{-1}$, and vice versa.
\end{definition}

\begin{definition}
A map $T\in \GL(n,\R)$ is said to be {\it proximal} if it has a unique (real) eigenvalue of maximal absolute value having algebraic (and hence geometric) multiplicity one.
\end{definition}
\smallskip
Suppose $f$ is a non-trivial (i.e., $f\neq \text{Id}$) orientation preserving homeomorphism of $\mbb{S}^1$. Then any fixed or periodic point of $f$ is attracting for $f$ or $f^{-1}$ unless some power of $f$ is identity, see \cite{HP}. An easy consequence of this fact is that any non-trivial homeomorphism $f$ admitting a fixed or a periodic point cannot be distal unless some power of $f$ is the identity map. For the homeomorphism $\T_a$ on $\mbb{S}^1$, Proposition \ref{prop2} below characterizes those $\T_a$ for which $\T_a^2=\text{Id}$.
In what follows, we show the existence of fixed or periodic points of order two for certain classes of $\T_a$, and then we use Proposition \ref{prop2} (sometimes intrinsically) to conclude that $\T_a$ is not distal.   
\begin{proposition}\label{prop2}
Let $T\in\GL(2,\R)$ and $a\in\R^2$ be such that $\|T^{-1}(a)\|<1$. Then for the homeomorphism $\T_a$ on $\mbb{S}^1$, $\T_a^2=\text{Id}$ if and only if the following conditions are satisfied: 
\begin{itemize}
    \item[$(i)$] $a$ is an eigenvector of $T$ corresponding to a real negative eigenvalue $\lambda_1$.
    \smallskip
    
    \item [$(ii)$] The other eigenvalue $\lambda_2$ is different from $\lambda_1$ and the eigenvector corresponding to $\lambda_2$ is orthogonal to $a$, and
    \smallskip
    
     \item[$(iii)$] $\lambda_1^2=\|a\|^2+\lambda_2^2$. 
\end{itemize}
\end{proposition}
\begin{proof}
If $\T_a^2=\text{Id}$, then $\T_a^2(x)=x$ and $\T_a^2(-x)=-x$, for any $x\in\mbb{S}^1$. 
From the definition of $\T_a$ we have
\begin{eqnarray}
b_1a+T(a)+T^2(x) &=b_2 x\\
b'_1a+T(a)-T^2(x) &=-b'_2 x,
\end{eqnarray}
where $b_1=\|a+T(x)\|$, $b'_1=\|a-T(x)\|$, $b_2=\|b_1a+T(a)+T^2(x)\|$ and $b'_2=\|b'_1a+T(a)-T^2(x)\|$. By adding and subtracting the above equations we get 
\begin{eqnarray}\label{iv}
(b_1+b'_1)a+2T(a)=(b_2-b'_2)x
\end{eqnarray} 
and
\begin{eqnarray}\label{v}
(b_1-b'_1)a+2T^2(x)=(b_2+b'_2)x,
\end{eqnarray} respectively for every $x\in\mbb{S}^1$.
\smallskip

Suppose $a$ is not an eigenvector of $T$. Then $b_2- b'_2\neq 0$ for all $x\in\mbb{S}^1$, otherwise $a$ would become an eigenvector of $T$ by \eqref{iv}. But $b_2-b'_2\neq 0$ for all $x\in\mbb{S}^1$ ensures that either $x$ or $-x$ belongs to the positive cone generated by $a$ and $T(a)$ in $\R^2$ for all $x$ in $\mbb{S}^1$ which is not possible.
 \smallskip

Now suppose $a$ is an eigenvector of $T$, \textit{i.e.}, $T(a)=\lambda_1 a$ for some $\lambda_1\in\R\setminus\{0\}$. If $\lambda_1>0$, then it follows from $\eqref{iv}$ that $b_2-b'_2\neq 0$ for all $x\in\mbb{S}^1$. But then each $x$ in $\mbb{S}^1$ lies on a straight line which is not possible. Now suppose $\lambda_1<0$. Observe that it is always possible to choose $y\in\mbb{S}^1$ such that the vector $a$ and the vector $T(y)$ are perpendicular to each other. In that case $b_1=b'_1$ at $x=y$, and then it follows from \eqref{v} that $y$ is the second eigenvector of $T$ corresponding to the eigenvalue $\lambda_2$ (say), where $\lvert \lambda_2\rvert=\sqrt{(b_2+b'_2)/2}$ at $x=y$. Therefore, $y$ and $a$ are also orthogonal to each other. Let $\overline{a}=\frac{a}{\|a\|}$, then a straightforward calculation shows that 
\begin{align*}
(b_1+b'_1)  &=2\lvert\lambda_1\rvert\ \text{at}\ x =\overline{a},\ \text{and}\\ 
 (b_1+b'_1) &=2\sqrt{\|a\|^2+\lambda_2^2}\ \text{at}\ x =y.
\end{align*}
So, if $\lambda_1^2 \neq \|a\|^2+\lambda_2^2$, then $b_1+b'_1$ takes two different values on $\overline{a}$ and $y$. Then it follows from \eqref{iv} that $b_2-b'_2\neq 0$ for some $x\in\mbb{S}^1$, otherwise, $a$ would become an eigenvector of $T$ corresponding to more than one distinct eigenvalues. But $b_2-b'_2\neq 0$ for some $x$ in $\mbb{S}^1$ implies that $b_2-b'_2\neq 0$ in a neighbourhood $U$ of $x$ in $\mbb{S}^1$ as well since $b_2-b'_2$ is a continuous function on $\mbb{S}^1$. Then by \eqref{iv} again, every $x\in U$ lie on a line generated by $a$ which is impossible. Thus we see that if $\T_a^2= \text{Id}$, then all the three conditions in the statement are satisfied.

\smallskip

Conversely, suppose all the three conditions are satisfied. 
Then a straightforward calculation shows that $\T_a^2(x)=x$, for every $x\in\mbb{S}^1$.
\end{proof}

\begin{remark}
From the Proposition \ref{prop2} above, we have that only for a self-adjoint proximal map $T$ where the dominant eigenvalue is negative, there exists a unique (up to a scaling by $\pm1$) $a$ for which $\T_a^2=\text{Id}$ and $\T_{-a}^2=\text{Id}$. In that case, $\T_a$ and $\T_{-a}$ both are distal. So, the class of $T$ and $a$ not satisfying those conditions is ``larger" than the class of those satisfying the conditions. In what follows, we show that $\T_a$ is not distal in many cases, and whenever we wish to show the non-distality of $\T_a$ by showing the existence of periodic points of order $2$, we have to avoid the cases in which the conditions of Proposition \ref{prop2} are satisfied.
\end{remark}

It is rather easy to deal with $\T_a$ when $T$ has real eigenvalues. It was shown in Theorem $7$ of \cite{SY1} that if $T$ has one positive real eigenvalue, then for any $a\in\R^2$ with $\|T^{-1}(a)\|<1$, $\T_a$ has a fixed point and so $\T_a$ is not distal. On the other hand if $T$ has two negative real eigenvalues, then it is easy to see that if $a$ is an eigenvector, $\bar a=\frac{a}{\|a\|}$ is a periodic point of order 2 of the map $\T_a$ (cf. Corollary $10$ of \cite{SY1}). So, in this case, if we choose the eigenvector $a$ so that, $a$ along with the negative eigenvalues $\lambda_1$ and $\lambda_2$, violates any of the conditions in Proposition \ref{prop2}, then $\T_a$ is not distal. 
We may summarize as follows:
\begin{theorem}\label{real}
Suppose $T\in\GL(2,\R)$ has real eigenvalues. Then $\T_a$ is not distal for any $a\in T\Big(B_1^*(0)\Big)$, where $B_1^*(0)$ is the reduced ball of radius $1$ around the origin, if $T$ has at least one positive eigenvalue. On the other hand, if $T$ has both the eigenvalues negative, then $\T_a$ is not distal for all $a$ on either of the line generated by the eigenvectors of $T$ satisfying $a\in T\Big(B_1^*(0)\Big)$, and not satisfying all the conditions of Proposition \ref{prop2} simultaneously.
\end{theorem}
\begin{remark}\label{additional}
In the proof of Corollary $10$ of \cite{SY1}, when both the eigenvalues of $T$ are real and negative, one has to choose $a$ as mentioned in Theorem \ref{real} above to ensure the non-distality of $\T_a$.
\end{remark}

Now we turn our attention to the cases in which $T$ has complex eigenvalues. In these cases, without loss of generality we may assume that $T$ is a rotation. We look for the existence of fixed or periodic points of the homeomorphism $\T_a$, and investigate the behaviour of the orbits of points near the fixed or periodic points.
In what follows, to specify a rotation by $\theta$ of any point on the unit circle, we will often use the multiplicative structure on the field of complex numbers. More precisely, if $T$ is a rotation by an angle $\theta$, then $Tx=rx$ for any point $x$ on the unit circle, where $r=(r_1,r_2)$ with $r_1=\cos\theta$ and $r_2=\sin\theta$. For any $x\in\mathbb{R}^2\backslash\{(0,0)\}$, we will also denote by $x^{-1}$ the multiplicative inverse of $x$ in $\mathbb{C}$.

\begin{proposition}\label{pro5}
Let $T\in\textnormal{SO}(2,\mathbb R)$ be a rotation, \text{i.e.} $T(x)=rx$ with $r=(\cos\theta, \sin\theta)$. Let $a\in\R^2\setminus\{(0,0)\}$ be such that $\|T^{-1}(a)\|<1$. Then $\overline{T}_a$ admits a fixed point if and only if $\cos\theta\geq\sqrt{1-\alpha^2}$, where $\alpha=\|a\|$. Moreover the fixed points of $\overline{T}_a$, whenever they exist, are of the form \begin{center}  $a(t-r)^{-1}$, where $t=\cos\theta\pm\sqrt{\alpha^2-\sin^2\theta}$.\end{center}
\end{proposition}

\begin{proof}
 Suppose $\overline{T}_a$ has a fixed point, say, $x_0$. Then 
 \begin{equation}
     \frac{a+rx_0}{\|a+rx_0\|}=x_0
 \end{equation}
i.e., $a+rx_0=tx_0$, where $t=\|{a+rx_0}\|$. By setting $r_1=\cos\theta$ and $r_2=\sin\theta$, $a$ can be rewritten as $a=(t-r_1, -r_2)x_0$, and then $t$ satisfies the quadratic equation 
\begin{equation}\label{quadeq}
    t^2-2tr_1+1-\alpha^2=0.
\end{equation} 
As $T$ is an isometry, it follows that $\alpha=\|a\|=\|T^{-1}(a)\|<1$. Since $t>0$, we have that $r_1>0$, otherwise, the equation \eqref{quadeq} cannot hold. By solving the equation \eqref{quadeq}, we obtain $t=r_1\pm\sqrt{r_1^2-(1-\alpha^2)}$. Since $r_1$ and $t$ are both positive real numbers, we see that $r_1\geq\sqrt{1-\alpha^2}$ and $|r_2|\leq\alpha$.
\smallskip

Conversely, suppose $r_1\geq\sqrt{1-\alpha^2}$, and consequently, $|r_2|\leq\alpha$. Since $r_1>0$, because $\alpha<1$, $t=r_1\pm\sqrt{r_1^2-(1-\alpha^2)}$ are positive real numbers for which $\|t-r\|=\alpha$. If we choose $x_t=a(t-r)^{-1}$, then $\|x_t\|=1$ and $\overline{T}_a(x_t)=x_t$. Note that, $\overline{T}_a$ has only one fixed point if $r=\left(\sqrt{1-\alpha^2}, \pm\alpha\right)$.
\end{proof}
\medskip

Next we investigate the existence or non-existence of the periodic points of the map $\T_a$ when $T$ is a rotation.
\begin{proposition}\label{lem6}
Let $T(x)=rx$, where $r=(\cos\theta, \sin\theta)$, and $a\in\R^2\setminus\{(0,0)\}$ be such that $\|T^{-1}(a)\|<1$. Let  $\|a\|=\alpha$ and $\T_a$ be as in \eqref{tabar}, then we have the following:
\begin{enumerate}
\item[$(i)$] If $\cos\theta>0$ and $\lvert\sin\theta\rvert>\alpha$, then $\overline{T}_a$ has no periodic point of order two. 
\item[$(ii)$] There exists a neighbourhood $U$ of $(-1,0)$ such that for $(\cos\theta,\sin\theta)\in U$, $\overline{T}_a$ has four periodic points of order two, and consequently it is not distal.
\end{enumerate}
\end{proposition}

\begin{proof}
Let us define $b_1:=\|a+rx\|$ and $b_2:=\|b_1a+ra+r^2x\|$. Then
\begin{equation*}
    \T_a(x)=\frac{a+r x}{b_1} \quad \text{ and } \quad \T_a^2(x)=\frac{b_1a+ra+r^2x}{b_2}
\end{equation*}

Now, it is easy to see that $\overline{T}^n_{sa}(sx)=s\overline{T}^n_a(x)$, for any $s\in\mbb{S}^1$ and $n\in\N$. Here, $sa$ stands for the multiplication of $s$ and $a$ as complex numbers, which also means a rotation of the vector $a$ by a certain angle depending on $s$. So, $x$ is a periodic point of $\T_a$ if and only if $sx$ is a periodic point of $\T_{sa}$. Therefore, if needed, multiplying $a$ by a suitable $s$, we may assume that $a$ is the vector $(0,\alpha)$ with $\alpha>0$. As $\|T^{-1}a\|<1$ and $T$ is a rotation, it follows that $0<\alpha<1$. Given two points $p$ and $q$ on $\mbb{S}^1$, we denote by $[p,q]$ the arc from $p$ to $q$ moving counterclockwise.

\medskip

\noindent {\em Proof of $(i)$.}
\begin{figure}[!ht]
     \centering
    \begin{tikzpicture}[scale=1.2, every node/.style={scale=0.8}]
    \definecolor{pallido}{RGB}{221,227,227}
    
    \draw [thick, violet] (0,0.5) arc [start angle = 90, end angle =20 , radius = 0.5];
    \draw[thin, gray] (-4,0)--(4.5,0);
    \draw[thin, gray] (0,4)--(0,-4);
    \draw[thick, black] (0,0) circle (3cm);
    \draw[thin, black, -latex] (0,0)--(0,2);
    \draw[thin, black, latex-] (0,-2)--(0,0);
    \draw[thin, black] (0,0)--(4.5,1.5);

    \fill (0,3) circle (1.5pt);
    \fill (2.8375,0.9375) circle (1.5pt);
    \fill (-2.12132034,2.12132034) circle (1.5pt);
    \fill (-2.975,0.375) circle (1.5pt);
    \fill (-0.375, -2.975) circle (1.5pt);
    \fill (-0.85, -2.869) circle (1.5pt);
    \fill (-2.82, -1.) circle (1.5pt);
    \fill (2.97, 0.375) circle (1.5pt);

    \node at (0.5,0.5) {$\theta$};
    \node at (0.25,2) {$a$};
    \node at (0.25,3.25) {$\overline a$};
    \node at (-0.25,-2) {$-a$};
    \node at (4.125,0.875) {$y=\overline T_a^{-1}(\overline a)=r^{-1}a$};
    \node at (-2.25,0.375) {$t=\overline T_a^2(\overline a)$};
    \node at (-2.75,2.25) {$w=\overline T_a(\overline a)$};
    \node at (-0.75,-3.25) {$z=\overline T_a^{-2}(\overline a)$};
    \node at (-1.9, -1.0){$p=\overline T_a(t)$};
    \node at (-0.72,-2.55) {$q=\overline T_a(p)$};
    \node at (2.25, 0.25) {$k=\overline T_a(q)$};
    \end{tikzpicture}
    \caption{}
    \label{fig1}
\end{figure}
The assumptions of $(i)$ mean that $\theta$ is an acute angle (positive or negative) and there is some lower bound on the absolute value of $\theta$. Without loss of generality, we may assume that $\theta>0$ since the case $\theta<0$ can be dealt with similarly. Let $\bar a:=\frac{a}{\|a\|}$. Then it is easy to see that $w:=\T_a(\bar a)$ lies in the $2^{nd}$ quadrant and $t:=\T_a^2(\bar a)$ lies either in the $2^{nd}$ or in the $3^{rd}$ quadrant (see Figure \ref{fig1}). On the other direction, $y:=\T_a^{-1}(\bar a)=r^{-1}\bar a$ lies in the $1^{st}$ quadrant, and $z:=\T_a^{-2}(\bar a)$ does not lie in the $2^{nd}$ quadrant (see Figure \ref{fig1}). It is also easy to see that if $t=(t_1,t_2)$ and $z=(z_1,z_2)$, then $z_1>t_1$. Then the points $z,y,\bar a,w,t$ are ordered counterclockwise. Note that $\T_a^2$ takes the arc $[z,y]$ to the arc $[\bar a,w]$ which is disjoint from $[z,y]$, and the arc $[y,\bar a]$ is mapped to $[w,t]$ by $\T_a^2$ which is again disjoint from $[y,\bar a]$. This shows that we can not have any periodic point of order $2$ in $[z,t]$. 
\smallskip

Now we show that there can not be any periodic point of order $2$ of $\T_a$ in $[t,z]$ as well.
Suppose $p:=\T_a(t)$. Then it is easy to check that either $p\in[t,z]$, or, $p\in[z,y]$. Let $p\in[t,z]$. Then $[t,p]$ is mapped to $[q,k]$ by $\T_a^2$ where $q:=\T_a(p)=\T_a^2(t)$, and $k:=\T_a(q)=\T_a^2(p)$. Also, $[p,z]$ is mapped to $[k,\bar a]$ by $\T_a^2$. 
As $[t,p]\cap[q,k]=\emptyset$ and $[p,z]\cap[k,\bar a]=\emptyset$, we see that $\T_a$ has no periodic point of order two in $[t,z]$.

Now suppose $p\in[z,y]$. In that case, we see that $[t,z]$ is mapped to $[q,\bar a]$ by $\T_a^2$. As $[t,z]$ and $[q,\bar a]$ are disjoint in this case (as $q$ will be again somewhere in the arc $[p,\bar a]$), it follows that there is no periodic point of order $2$. This completes the proof of $(i)$.

\medskip

\noindent {\em Proof of (ii).}
\begin{figure}[h]
    \centering
    \begin{tikzpicture}[scale=1.2, every node/.style={scale=0.8}]
    \definecolor{pallido}{RGB}{221,227,227}
    
    \draw[thin, gray] (-4,0)--(4.5,0);
    \draw[thin, gray] (0,4)--(0,-4);
    \draw[thick, black] (0,0) circle (3cm);
    \draw[thin, black, -latex] (0,0)--(0,2);

    \fill (0,3) circle (1.5pt);
    \fill (2.8375,0.9375) circle (1.5pt);
    \fill (-2.975,0.375) circle (1.5pt);
    \fill (2.975,0.375) circle (1.5pt);
    \fill (-0.375, -2.975) circle (1.5pt);
    \fill (0.375, 2.975) circle (1.5pt);
        
    \node at (0.25,2) {$a$};
    \node at (3.625,-0.25) {$A=(1,0)$};
    \node at (-3.5,-0.25) {$(-1,0)$};
    \node at (-0.625,3.25) {$B=(0,1)$};
    \node at (0.5,2.75) {$B'$};
    \node at (0.675,-3.25) {$C=(0,-1)$};
    \node at (3.125,0.875) {$x$};
    \node at (-2.5,0.375) {$\overline T_a(x)$};
    \node at (2.75,0.375) {$A'$};
    \node at (-0.5,-3.25) {$C'$};
    \end{tikzpicture}
    \caption{}
    \label{fig2}
\end{figure}
Let $r_1=\cos\theta$ and $r_2=\sin\theta$. Then by some straightforward calculations, we have that
\begin{equation*}
\overline{T}_a^2(1,0)=(x_1,x_2),\ \text{where}\ x_1=\frac{r_1^2-r_2^2-r_2\alpha}{b_2},\ x_2=\frac{b_1\alpha+r_1\alpha+2r_1r_2}{b_2}.\end{equation*}
\begin{equation*}
\overline{T}_a^2(0,1)=(y_1,y_2),\ \text{where}\ y_1=\frac{-2r_1r_2-r_2\alpha}{b_2},\ y_2=\frac{r_1^2-r_2^2+r_1\alpha+b_1\alpha}{b_2}.\end{equation*}
\begin{equation*}
\overline{T}_a^2(0,-1)=(z_1,z_2),\ \text{where}\ z_1=\frac{2r_1r_2-r_2\alpha}{b_2},\ z_2=\frac{r_2^2-r_1^2+r_1\alpha+b_1\alpha}{b_2}.
\end{equation*}
\smallskip

As $\alpha<1$, there exists a neighbourhood $U$ of $(-1,0)$ such that if $(r_1,r_2)\in U$ with $r_2>0$, then $x_i>0, y_i>0$ and $z_i< 0$ for $i=1,2$. 
Let us denote by $A,B,C$ the points $(1,0),(0,1)$ and $(0,-1)$ respectively as in Figure \ref{fig2}. Also let $A':=\T_a^2A,\ B':=\T_a^2B$ and $C':=\T_a^2C$. As $x_i>0$ and $y_i>0$ for $i=1,2$, we have that $A',B'\in [A,B]$. As $r_2$ is very small compared to the absolute value of $r_1$, it follows that $x_1>y_1$. This shows that the points $A,A',B',B$ are ordered counterclockwise. Now, there are two possibilities: $\T_a^2[A,B]=[A',B']\subseteq[A,B]$ or $\T_a^2[A,B]=[B',A']$. As $C,C'\in[B,A]$, the second possibility is impossible. Hence, the arc $[A,B]$ is mapped into itself by $\T_a^2$, and consequently $\T_a^2$ has a fixed point in $[A,B]$. Suppose $x\in[A,B]$ be such that $\T_a^2(x)=x$. Then $\T_a^2(\T_a(x))=\T_a(x)$ as well. As $r=(\cos\theta,\sin\theta)$ is very close to $(-1,0)$, it follows that $\theta$ is very close to $\pi$ (but less than $\pi$ as $r_2=\sin\theta>0$). Then it is easy to see that $\T_a(x)$ does not lie in the $1^{st}$ quadrant. So, we get two fixed points of $\T_a^2$. 

In a similar manner, considering the arc $[C,A]$, we will get two fixed points $y$ and $\T_a(y)$ (say) of $\T_a^2$ again. By a similar argument as above we can say that $y$ and $\T_a(y)$ are distinct. Also it is easy to see that $\T_a(x)$ does not lie in the $4^{th}$ quadrant. Hence, $y$ and $\T_a(x)$ are distinct. Then it follows that $x$, $\T_a(x)$, $y$ and $\T_a(y)$ are four distinct periodic points of order $2$ of the map $\T_a$. The case when $r_2<0$ can be treated similarly to get four periodic points of order $2$ for the map $\T_a$ as well. This completes the proof.
\end{proof}

Now we determine the nature of fixed or periodic points of $\T_a$ by looking at the behavior of orbits of the points near the fixed or periodic points.

\begin{theorem}\label{th7}
Let $T$ be a rotation defined by $T(x)=rx$ with $r=(\cos\theta,\sin\theta)$, and $a\in\R^2\setminus\{(0,0)\}$ be such that $\alpha=\|a\|<1$. Then, for $\cos{\theta}>\sqrt{1-\alpha^2}$, $\overline{T}_a$ has two fixed points; one of them is attracting and the other is repelling. 
\end{theorem}
\begin{proof}

\begin{figure}[h]
    \centering
        \begin{tikzpicture}[scale=1, every node/.style={scale=0.8}]
    \definecolor{pallido}{RGB}{221,227,227}
    
    \draw[thin, gray] (-4,0)--(4.5,0);
    \draw[thin, gray] (0,4)--(0,-4);
    \draw[thick, black] (0,0) circle (3cm);
    \draw[thin, black, -latex] (0,0)--(0,2);
    \draw[thin, black] (0,0)--(-1.5, 2.5980762);
    \draw[thin, black] (0,0)--(-1.5, -2.5980762);
       \draw[thick, black, ->] (3,0) arc [start angle = 0, end angle =15 , radius = 3];
    \draw[thick, black, ->] (-3,0) arc [start angle = 180, end angle =175 , radius = 3];
    
    \fill (3,0) circle (1.5pt);
    \fill (-3,0) circle (1.5pt);
    \fill (-1.5, 2.5980762) circle (1.5pt);
    \fill (-1.5, 2.5980762) circle (1.5pt);
    \fill (-1.5, -2.5980762) circle (1.5pt);
    \fill (2.5980762, 1.5) circle (1.5pt);
    \fill (-2.975,0.375) circle (1.5pt);
        
    \node at (0.25,2) {$a$};
    \node at (3, 1.5) {$A'$};
    \node at (4,-0.25) {$A=(1,0)$};
    \node at (-4,-0.25) {$D=(-1,0)$};
    \node at (-3.25, 0.375) {$D'$};
    \node at (-1.5, 2.85) {$Q$};
    \node at (-1.5,-2.85) {$P$};

    \end{tikzpicture}
    \caption{}
    \label{fig3}
\end{figure}
 As in the above Proposition, without loss of any generality, we assume that $a=(0,\alpha)$ for some $\alpha>0$. As $T$ is a rotation, $\alpha<1$ as well. Now let $r_1=\cos\theta$ and $r_2=\sin\theta$. As $r_1>\sqrt{1-\alpha^2}$, $|r_2|<\alpha$. We recall from Proposition \ref{pro5} that the fixed points of $\overline{T}_a$ are given by $a(t-r)^{-1}$, where $t=\cos{\theta}\pm\sqrt{\alpha^2-\sin^2{\theta}}$. An easy calculation shows that these fixed points are given by 
\[\Big(-\frac{r_2}{\alpha}, -\frac{\sqrt{\alpha^2-r_2^2}}{\alpha}\,\Big) \text{ and } \Big(-\frac{r_2}{\alpha}, \frac{\sqrt{\alpha^2-r_2^2}}{\alpha}\,\Big),
\]
which we denote by $P$ and $Q$ respectively (see Figure \ref{fig3}). Note that $P$ and $Q$ are conjugate to each other. Also let $A:=(1,0)$ and $D:=(-1,0)$ (see Figure \ref{fig3}); $A':=\T_aA$ and $D':=\T_aD$.  It is easy to see that 
\begin{equation*}
A'=\T_a(1,0)=
\frac{1}{\sqrt{1+\alpha^2+2\alpha r_2}}(r_1,\alpha+r_2), \end{equation*}
and
\begin{equation*}
D'=T_a((-1,0))=
\frac{1}{\sqrt{1+\alpha^2-2\alpha r_2}}(-r_1,\alpha-r_2). \end{equation*}
Since $P$ and $Q$ are the only fixed points of $\T_a$, there are two possibilities: the arc $[P,Q]$ (taken in the sense of counterclockwise) mapped either to the arc $[P,Q]$ or to the arc $[Q,P]$ by the map $\T_a$. Now we see that the $1^{st}$ co-ordinate of both $A$ and $A'$ are positive, and the $1^{st}$ co-ordinates of both $D$ and $D'$ are negative. Then it follows that $\T_a$ does not interchange the arcs $[P,Q]$ and $[Q,P]$, and consequently, the arc $[P,Q]$ is mapped to itself by $\T_a$. If we look at the $2^{nd}$ co-ordinates of $A'$ and $D'$, we see that they both are positive (both are zero for $A$ and $D$). Therefore, $A'$ is closer to $Q$ than $A$, and $D'$ is closer to $Q$ than $D$. We conclude that $Q$ is the attracting fixed point of $\T_a$, and $P$ is the repelling one.
\end{proof}
\medskip

Note that if $T$ is the identity map, then $\overline{a}$ and $-\overline{a}$ are the fixed points of $\T_a$ with $\overline{a}$ being the attracting one and $-\overline{a}$ being the repelling one. In the next Proposition, we discuss the dynamics of $\T_a$ when $T=-\text{Id}$. Note that Proposition \ref{lem6} and Theorem \ref{th7} do not cover this case.

\begin{proposition}\label{cor8}
Let $T=-\text{Id}$ and $a\in{\R}^2\setminus\{(0,0)\}$ be such that $\|T^{-1}(a)\|<1$. Then 
 $\overline{a},\ -\overline{a},\ x_0$ and $a-x_0$ are the only periodic points of $\overline{T}_a$ of order two, where $x_0$ is such that $\|a- x_0\|=1$. Moreover, $x_0$ and $a-x_0$ are the attracting fixed points, whereas, $\overline{a}, -\overline{a}$ are the repelling fixed points. Consequently, $\T_a$ is not distal.
\end{proposition}

\begin{proof}

\begin{figure}[h]
    \centering
    \begin{tikzpicture}[scale=1, every node/.style={scale=0.8}]
    \definecolor{pallido}{RGB}{221,227,227}
    
    \draw[thin, gray] (-4,0)--(4.5,0);
    \draw[thin, gray] (0,4)--(0,-4);
    \draw[thick, black] (0,0) circle (3cm);
    \draw[thin, black, -latex] (0,0)--(0,2);

    \fill (0,3) circle (1.5pt);
    \fill (0,-3) circle (1.5pt);
    \fill ( 2.8375,0.9375) circle (1.5pt);
    \fill (-2.8375,0.9375) circle (1.5pt);
        
    \node at (0.25,2) {$a$};
    \node at (3.5,-0.25) {$(1,0)$};
    \node at (-3.5,-0.25) {$(-1,0)$};
    \node at (-0.25,3.25) {$\overline a$};
    \node at (0.25,-3.25) {$-\overline a$};
    \node at (3.375,0.875) {$a-x_o$};
    \node at (-3.125,0.875) {$x_o$};
    \end{tikzpicture}
    \caption{}
    \label{fig4}
\end{figure}

Again without loss of generality, we assume that $a=(0,\alpha)$ with $0<\alpha<1$. Now, it is always possible to choose $x_0$ on $\mbb{S}^1$ with $2^{nd}$ co-ordinate positive such that $\|a-x_0\|=1$, and $a-x_0$ has $2^{nd}$ co-ordinate positive as well (see Figure \ref{fig4}). Then some simple calculations show that $\T_a(x_0)=a-x_0$ and $\T_a(a-x_0)=x_0$; also $\T_a(\bar a)=-\bar a$ and $\T_a(-\bar a)=\bar a$. It follows that these four points $x_0,\ a-x_0,\ \bar a$ and $-\bar a$ are periodic points of order $2$ of the map $\T_a$.
Also $\T_a^2$ leaves the four consecutive arcs $[x_0,-\bar a],\ [-\bar a,a-x_0],\ [a-x_0,\bar a]$ and $[\bar a,x_0]$ invariant (see Figure $4$). So, in each of the arcs, there is one attracting and one repelling point. Again an easy calculation confirms that both $\T_a^2(1,0)$ and $\T_a^2(-1,0)$ have positive $2^{nd}$ co-ordinates. This says in particular that $-\bar a$ is the repelling point in both the arcs $[-\bar a,a-x_0]$ and $[x_0,-\bar a]$. The dynamics of $\T_a^2$ in the arc $[a-x_0,x_0]$ is conjugated, by $\T_a$, to the dynamics of $\T_a^2$ in the arc $[x_0,a-x_0]$, and $-\bar a$ is mapped to $\bar a$ by the conjugation map $\T_a$. Hence, $\bar a$ is a repelling point for $\T_a^2$ as well. Then the other two points $x_0$ and $a-x_0$ are the attracting points  of $\T_a^2$. 
\end{proof}


\section{Higher dimensional spheres}
\smallskip

We turn our attention now to the dynamics of $\T_a$ on higher dimensional spheres. In \cite{SY1}, it was shown that if $T\in \GL(n+1,\R)$ is an isometry of $\R^{n+1}$ and $n$ is even, then there exists non-zero $a$ in $\R^{n+1}$ such that $\T_a$ on $\mbb{S}^n$ is not distal. In this article, we exhibit the non-distality for larger classes of $\T_a$. Before doing that, let us recall a definition which will be useful in what follows.

\begin{definition}
For $T\in \GL(n,\R)$, the $T$-invariant subspace 
\begin{equation}
    C(T):=\{x\in \R^n\mid T^n(x)\to 0\ as\ n\to\infty\}
\end{equation}
is called the \textit{contraction space} of $T$. 
\end{definition}

Note that, for $T\in \GL(n,\R)$, $T\arrowvert_{C(T)}$ has all the eigenvalues of absolute value less than $1$, and $T\arrowvert_{C(T^{-1})}$ has all the eigenvalues of absolute value greater than $1$. 
\begin{theorem}\label{thm:distcomplexeigen}
Let $T\in \GL(n+1,\R)$, and one of the following holds:
\begin{enumerate}
\item[$(i)$] $T$ has two real eigenvalues or $T$ has a complex eigenvalue of the form $t(\cos{\theta}+i\sin{\theta})$, $t>0$ such that $0<\cos{\theta}<1$.
\item[$(ii)$] $T$ is proximal with $\det T>0$.
\end{enumerate}
Then there exists an $a\in\R^{n+1}$ with $0<\|T^{-1}(a)\|<1$ such that $\overline{T}_a$ has a fixed point or a periodic point of order two in $\mbb{S}^n$, and consequently, $\ol{T}_{a}$ is not distal. 
\end{theorem}

\begin{proof} Suppose (i) holds. Then $T$ keeps a $2$-dimensional subspace $V$ (say) invariant. Now consider the restriction of $T$ to $V$. By Theorem \ref{real} and the discussion before Theorem \ref{real} (if $T$ has two real eigenvalues) and Corollary $10$ of \cite{SY1} (if $T$ has complex eigenvalues as in the statement), there exists $a\in V$ with $\|T^{-1}(a)\|<1$ such that the restriction of $\T_a$ to $V\cap\mbb{S}^n$ has either a fixed point or a periodic point of order $2$. Hence, $\T_a$ has a fixed point or a periodic point of order $2$ on $\mbb{S}^n$ and it is not distal.
\smallskip

Now suppose (ii) holds, i.e., $T$ is proximal and $\det T>0$. Then either $T$ has two distinct real eigenvalues or the only real eigenvalue $\lambda$ is positive. In the first case, the assertion follows from (i). In the second case, we have a 3-dimensional subspace $W$ (say) which is $T$-invariant. It is enough to show the assertion for 
$T\arrowvert_{W}$ as that is also proximal with a positive determinant. Without loss of any generality, we may assume that $n=2$. Since $\lambda$ is dominant and also positive, 
replacing $T$ by $T/\lambda$, we may assume 
that $\lambda=1$ and the restriction of $T$ to a $2$-dimensional subspace $V'$ (say) has eigenvalues of absolute value less than $1$, i.e., $V'=C(T)$. Let us 
take $a\ne 0$ such 
that $T(a)=a$ with $\|a\|<1$. Then $\T_a(\ol{a})=\ol{a}$ where $\ol{a}=\frac{a}{\|a\|}$. Let $a_0$ be such that $\|a_0\|=1$ and $\langle a_0,x\rangle=0$ for all $x\in V'$. Then $a_0$ and $V'$ generate the whole space. Let $y\in V'$ be such that $a=\alpha_0(a_0)+y$, where $\alpha_0>0$ as $a\not\in V'$ and $\alpha_0\leq \|a\|<1$. 
For any $x$ in $V'=C(T)$ with $\|x\|=1$, we have \begin{equation*}\ol{T}_a(x)=\frac{a+T(x)}{\|a+T(x)\|}.\end{equation*}
Let $\alpha_1=\|a+T(x)\|=\|\alpha_0a_0+y+T(x)\|\geq \alpha_0$. 
Similarly, for $m\geq 2$, we have \begin{equation*}\ol{T}_a^m(x)=\frac{s_ma+T^m(x)}{\|s_ma+T^m(x)\|},\end{equation*} where $s_1=1$ and 
\[ s_m=1+\sum_{i=1}^{m-1} \alpha_i>1,
\] for $\alpha_i=\|s_ia+T^i(x)\|=\|s_i(\alpha_0a_0+y)+T^i(x)\|\geq\alpha_0$, for all $i$. 
Therefore, 
\[ s_m=1+\sum_{i=1}^{m-1}\alpha_i\geq 1+(m-1)\alpha_0\to\infty \text{ as } m\to \infty.
\] Moreover, $T^m(x)\to 0$ as $m\to\infty$ as $x\in C(T)$. Therefore, we get that \begin{equation*}\ol{T}_a^m(x)=\frac{s_ma+T^m(x)}{\|s_m a+T^m(x)\|}\to \frac{a}{\|a\|}=\ol{a}\ \text{as}\ m\to\infty. \end{equation*} Since this holds for every $x\in V'$ with $\|x\|=1$, we have that $\ol{T}_a$ is not distal. In fact, if we take any point $z\in \mbb{S}^2$ which is a positive linear combination of $a$ and some $y\in V'$ with $\|y\|=1$, it is easy to 
show that $\ol{T}_a^m(z)\to a/\|a\|$.  
\end{proof}
Now we are in a position to say that given any $T\in \GL(n,\R)$, we can always find some element $S$ in its conjugacy class, and some $a$ in $\R^{n+1}$, such that $\ol{S}_a$ is not distal. A similar statement holds if we replace the conjugates of $T$ by certain powers of $T$. More precisely, we have the following:

\begin{corollary} \label{cor15}
For $T\in \GL(n+1,\R)$, the following statements hold: 
\begin{enumerate} 
\item [$(i)$] There exist a conjugate $S$ of $T$ in $\GL(n+1,\R)$ and $a\in\R^{n+1}\setminus\{0\}$ such that $\|S^{-1}(a)\|<1$, and $\ol{S}_a$ on $\mbb{S}^n$ is 
not distal. 
\item [$(ii)$] There exists $a\in \R^{n+1}\setminus\{0\}$ and some $S\in\{T,T^2,T^3\}$ such that $\|S^{-1}(a)\|<1$, and $\ol{S}_a$ on $\mbb{S}^n$ is not distal. 
\end{enumerate}
\end{corollary}

\begin{proof} For $T$ as above, either $T$ has two real eigenvalues or a complex eigenvalue of the form $t(\cos\theta+i\sin\theta)$, $t>0$. In the first case, both the assertions 
follows from Theorem \ref{thm:distcomplexeigen} (i) for $S=T$. In the second case suppose $0<\cos\theta<1$, then both the assertions follows from Theorem \ref{thm:distcomplexeigen} (i) for $S=T$. Now suppose 
$\cos\theta\leq 0$. As $T$ keeps a two dimensional space $V$ invariant, we can replace $T$ by $T\arrowvert_V$ and assume that $n=1$. Now $(i)$ follows from Corollary 10 and the Remark 11 of \cite{SY1}. For the second assertion, if $\cos\theta=0$, then $T^2$ has two identical real eigenvalues equal to $-1$, and if $\cos\theta<0$, then either $\cos(2\theta)>0$ or
$\cos(3\theta)>0$. In either of these cases, $(ii)$ follows for $S\in\{T,T^2,T^3\}$ from $(i)$ of Theorem \ref{thm:distcomplexeigen}. 
\end{proof}

Now we turn our attention to another aspect of this paper, \textit{i.e.}, expansivity of the homeomorphism $\T_a$.
As mentioned earlier, there is no expansive homeomorphism on $\mathbb{S}^n$ for $n=1,2$. But we are not aware of any such results on higher dimensional spheres. We show here that $\T_a$ is not expansive on $\mathbb{S}^n$, $n\geq 3$, for certain classes of $\T_a$.

\begin{theorem}\label{thm 15}
For any $T\in\GL(n+1,\R)$, $n\geq 1$, there exist uncountably many non-zero $a\in\R^{n+1}$ such that $\T_a$ is not an expansive homeomorphism on the $n$-dimensional sphere.
\end{theorem}

\begin{proof}
Suppose $T\in\GL(n+1,\R)$, $n\geq 2$. Then at least one of the following holds:
\begin{enumerate}
\item[]$(i)$ $T$ has at least two real eigenvalues or no real eigenvalue,
\item[]$(ii)$ $T$ has only one real eigenvalue.
\end{enumerate}
In either case, $T$ keeps a $2$-dimensional subspace $W$ (say) invariant. 
We know if $||T^{-1}(a)||<1$ then $\T_a$ is a homeomorphism, and $||T^{-1}(a)||<1$ if \begin{equation} \label{eq8} 0<\|a\|<\frac{1}{\|T^{-1}\|}.\end{equation} But one can find uncountably many $a\in W$ such that (\ref{eq8}) holds. For these choices of $a$'s, $\T_a$ is a homeomorphism of $\mathbb{S}^n$ and its restriction to $W\cap\mathbb{S}^n$ can be seen as a homeomorphism of $\mathbb{S}^1$. As there is no expansive homeomorphism on $\mbb{S}^1$, it follows that given any $\delta>0$ there exist $x,y\in W\cap\mbb{S}^n$ with $x\neq y$, such that \begin{equation*}\|(\T_a^n(x), \T_a^n(y))\|<\delta\ \text{for all}\ n\in\Z. \end{equation*} This implies that for a given $\delta>0$ one can always find distinct $x,y\in\mbb{S}^n$ such that \begin{equation*}\|(\T_a^n(x), \T_a^n(y))\|<\delta\ \text{for all}\ n\in\Z. \end{equation*} 
Hence, $\T_a$ can not be an expansive homeomorphism of $\mbb{S}^n$. This completes the proof.
\end{proof}
\begin{remark}
Even if we take $a$ to be the zero vector and define $\T(x)=\frac{T(x)}{\|T(x)\|}$ for $T\in\GL(n+1,\R)$, then using the same proof as in the above theorem, it can be shown that $\T$ is not expansive on $\mbb{S}^n$.
\end{remark}
The following proposition shows that using the map $\T_a$ on $\mbb{S}^n$, it is easy to produce examples of non-distal and non-expansive maps on products of spheres, and in particular, on $n$-dimensional torus by seeing it as an $n$-fold product of $\mbb{S}^1$. 

\smallskip

\begin{proposition}\label{prop:prodist}
Let $X_1$ and $X_2$ be compact metric spaces and $T_1$, $T_2$ be homeomorphisms of $X_1$ and $X_2$ respectively. Let $T\colon X_1\times X_2\to X_1\times X_2$ be the cartesian product of $T_1$ and $T_2$, i.e.,  $$T(x_1,x_2)=(T_1x_1,T_2x_2),$$ for $x_1\in X_1$ and $x_2\in X_2$. Then
\begin{enumerate}
\item[]$(i)$ $T$ is distal if and only if both $T_1$ and $T_2$ are distal.
\item[]$(ii)$ $T$ is expansive if and only if both $T_1$ and $T_2$ are expansive.
\end{enumerate}
\begin{proof}
$(i)$ Suppose $T_1$ and $T_2$ both are distal but $T$ is not distal. Then there exist $(x_1,x_2),\ (y_1,y_2)\in\ X_1\times X_2$ with $x_1\neq y_1$ (say) such that $$\bigl((x,y),\ (x,y)\bigr)\in \overline{\Bigl\{\Bigl(T^n(x_1,x_2),\ T^n(y_1,y_2)\Bigr)\Bigr\}}$$ for some $x\in X_1$ and $y\in X_2$. Since $T^n(x_1,x_2)=(T_1^nx_1,T_2^nx_2)$, it follows that $(x,x)\in \overline{\{(T_1^nx_1,T_1^ny_1)\}}$ contradicting the fact that $T_1$ is distal.

\smallskip 

Conversely, suppose that $T$ is distal but $T_1$ is not distal. Then there exist $x_1,y_1$ in $X_1$ with $x_1\neq y_1$ such that $(x,x)\in \overline{\{(T_1^nx_1,T_1^ny_1)\}}$ for some $x\in X_1$. Choose some $x_2$ in $X_2$ and let $y\in \overline{\{T_2^nx_2\}}$. Then it follows that $$\bigl((x,y),(x,y)\bigr)\in\ \overline{\Bigl\{\Bigl((T_1^nx_1,T_2^nx_2),\ (T_1^ny_1,T_2^nx_2)\Bigr)\Bigr\}}=\overline{\Bigl\{\Bigl(T^n(x_1,x_2),\ T^n(y_1,x_2)\Bigr)\Bigr\}}$$ which contradicts the distality of $T$.  

\smallskip

$(ii)$ We use the same notation $d$ to denote the metrics on $X_1$, $X_2$ and the product metric on $X_1\times X_2$. Now suppose $T_1$ and $T_2$ both are expansive with expansive constant $\delta_1$ and $\delta_2$ respectively, but, $T$ is not expansive. Then for any $\delta>0$, there exist $x_1,y_1\in X_1$ and $x_2,y_2\in X_2$ with at least $x_1\neq y_1$ such that 
$$d\Bigl(T^n(x_1,x_2),\ T^n(y_1,x_2)\Bigr)=d\Bigl((T_1^nx_1,T_2^nx_2),\ (T_1^ny_1,T_2^nx_2)\Bigr)<\delta$$ for all $n\in\Z$. But this contradicts the expansiveness of $T_1$ if we choose $\delta$ to be sufficiently small.
\smallskip

Conversely, suppose that $T$ is expansive with expansive constant $\varepsilon>0$, but, $T_1$ is not expansive. Then for any $\delta>0$, there is a pair $(x_1,y_1)$ of distinct points in $X_1$ such that $$d(T_1^nx_1,\ T_1^ny_1)<\delta$$ for all $n$ in $\Z$. But this implies $$d\Bigl(T^n(x_1,y),\ T^n(y_1,y)\Bigr)=d\Bigl((T_1^nx_1,T_2^ny),\ (T_1^ny_1,T_2^ny)\Bigr)<\delta$$ for all $n$ in $\Z$ which contradicts the expansivity of $T$ if we choose $\delta$ to be sufficiently small.
\end{proof}
\end{proposition}

Let $T_{k}\in \GL(i_k+1,\R)$ for $k=1,2,\dots,n$, and let $N=\sum\limits_{k=1}^n\ (i_k+1)$. We write
\begin{equation}\label{DT}
    \text{diag}\ [T_{1},\dots,T_{n}]=
    \begin{pmatrix}
    T_1 & 0 & \cdots & 0\\
    0 & T_2 & \cdots & 0\\
    0 & 0 & \dots & 0\\
    0 & 0 & \dots & T_n
    \end{pmatrix}
\end{equation}
for the $N\times N$ block diagonal matrix with blocks being $T_{1},\dots,T_{n}$. For any vector $\uu=(u_1,\dots, u_n)\in\R^N$, where $u_k\in \R^{i_k+1}$, let $\overline{T}_u$ be the map defined on $\mbb{S}^{i_1}\times\dots\times\mbb{S}^{i_n}$ as the product of $(\overline{T}_{1})_{u_1},\dots, (\overline{T}_{n})_{u_n}$, \textit{i.e.} $$\overline{T}_u=(\overline{T}_{1})_{u_1}\times\ \dots\ \times (\overline{T}_{n})_{u_n};$$ that is
\begin{equation}\label{aff}
    \T_{u}(\vv)=\Big((\T_1)_{u_1}(v_1),\dots, (\T_n)_{u_n}(v_n) \Big),
\end{equation}

\noindent for any $\vv=(v_1,\dots,v_n)\in\R^N$. Then it is easy to see that $\T_{u}$ is a homeomorphism if and only if $\|T_i^{-1}(u_i)\|<1$ for each $i=1,\dots,n$. The following results can be seen as easy consequences of Theorem \ref{thm:distcomplexeigen} and Proposition \ref{prop:prodist}.

\begin{proposition}
Let $\textnormal{diag}\ [T_1,\dots,T_n]$ be a block-wise diagonal matrix as in \eqref{DT} and let $u=(u_1,\dots,u_n)\in\R^N$ be such that $\|T_i^{-1}(u_i)\|< 1$ for each $i=1, \dots, n$. Then $\T_{u}$ is distal (resp.\ expansive) if and only if $\T_{u_i}$ is distal (resp.\ expansive) for all $i=1,\dots,n$.
\end{proposition}

\begin{theorem}
Suppose $T=\ \text{diag}\ [T_{1},\dots,T_{n}]\in\GL\left(N,\R\right)$ be such that it has two real eigenvalues or a complex eigenvalue of the form $t(\cos\theta+i \sin\theta),t>0$ such that $0<\cos\theta<1$ or $T_{i_k}$ is proximal with $\det(T_{i_k})>0$ for some $k=1,\dots,n$. Then there exists $u=(u_1,\dots, u_n)\in \R^N$ with $u_k\in\R^{i_k+1}$ for $k=1,2,\dots,n$ such that $\overline{T}_u$ on $\mbb{S}^{i_1}\times\dots\times\mbb{S}^{i_n}$ is not distal.
\end{theorem}

\begin{theorem}
For $T=\ \text{diag}\ [T_{i_1},\dots,T_{i_n}]\in\GL\left(N,\R\right)$, there exist uncountably many $u=(u_1,\dots u_n)\in\R^N$ with $u_k\in\R^{i_k+1}$ for $k=1,2,\dots,n$ such that $\overline{T}_u$ on $\mbb{S}^{i_1}\times\dots\times\mbb{S}^{i_n}$ is not expansive.
\end{theorem}


\medskip

\textbf{Acknowledgments} This work started at the International Centre for Theoretical Sciences (ICTS) when the authors went there for the programme Probabilistic Methods in Negative Curvature \url{ICTS/pmnc2019} and the authors are grateful for the hospitality of the ICTS and the nice environment they provide during any programme. The authors are also thankful to ICTS for organizing the second programme with the same theme Probabilistic Methods in Negative Curvature \url{ICTS/pmnc2021/3} which was also helpful to the authors. 
\smallskip

G. Faraco wish to thank Ursula Hamenst\"adt for invaluable support during the last six months of his period at Bonn. One of the authors A.\ K.\ Yadav would like to thank R.\ Shah for introducing this problem and having rigorous discussions during his Ph.D.\ at JNU, New Delhi and later, and also help in proving the Proposition \ref{prop2}. A.\ K.\ Yadav would also thank National Board for Higher Mathematics, India for Post-doctoral fellowship, and T.\ Das for local hospitality in the Department of Mathematics, University of Delhi during this fellowship. 
\smallskip

The authors are very much thankful to the referee for many insightful suggestions which led to a significant improvement in the presentation of the paper.

\end{document}